\documentclass[psamsfonts]{amsart}

\usepackage{amsmath, amsthm, amsfonts, amssymb, graphicx, url}

\makeatletter 
\newtheorem*{rep@theorem}{\rep@title}
\newcommand{\newreptheorem}[2]{%
\newenvironment{rep#1}[1]{%
 \def\rep@title{#2 \ref{##1}}%
 \begin{rep@theorem}}%
 {\end{rep@theorem}}}
\makeatother

\newtheorem{theorem}{Theorem}
\newtheorem{lemma}[theorem]{Lemma}
\newreptheorem{lemma}{Lemma}
\newreptheorem{theorem}{Theorem}

\newtheorem{cor}{Corollary}

\theoremstyle{definition}
\newtheorem{definition}{Definition}

\newtheorem{rem}{Remark}

\newcommand{\del}{\partial}
\newcommand{\abs}[1]{\left\lvert #1 \right\rvert}
\newcommand{\co}{\colon\thinspace}
\DeclareMathOperator{\Area}{Area}
\DeclareMathOperator{\const}{const}
\DeclareMathOperator{\Vol}{Vol}
\DeclareMathOperator{\dist}{dist}

\DeclareMathOperator{\Slab}{Slab}

\newcommand\Hy{\mathbb{H}}

\newcommand\Orb{\mathcal{O}}
\newcommand\isom{\operatorname{Isom}}

\newcommand\N{\mathbb{N}}
\newcommand\Z{\mathbb{Z}}
\newcommand\Q{\mathbb{Q}}
\newcommand\R{\mathbb{R}}

\begin{document}

\title[Thickness of orbifold skeletons]{Thickness of skeletons of \\arithmetic hyperbolic orbifolds}
\author{Hannah Alpert}
\address{University of British Columbia, 1984 Mathematics Road, Vancouver, BC, Canada}
\email{hcalpert@math.ubc.ca}
\author{Mikhail Belolipetsky}
\address{IMPA, Estrada Dona Castorina, 110, 22460-320 Rio de Janeiro, Brazil}
\email{mbel@impa.br}

\begin{abstract}
We show that closed arithmetic hyperbolic $3$--dimen\-sional orbifolds with larger and larger volumes give rise to triangulations of the underlying spaces whose $1$--skeletons are harder and harder to embed nicely in Euclidean space. To show this we generalize an inequality of Gromov and Guth to hyperbolic $n$--orbifolds and find nearly optimal geodesic triangulations of arithmetic hyperbolic $3$--orbifolds.

\end{abstract}

\maketitle

\section{Introduction}

Consider a closed hyperbolic $n$--dimensional manifold $X$ with $n \ge 3$. In \cite{Gromov12}, Gromov and Guth found a remarkable inequality which relates the hyperbolic volume of $X$, its isoperimetric Cheeger's constant, and retraction thickness of an embedding of $X$ into $\R^N$. Our first result is a generalization of this inequality to hyperbolic orbifolds, but with combinatorial thickness instead of retraction thickness. 

We define a closed \textit{\textbf{piecewise hyperbolic pseudomanifold}} to be a finite simplicial complex in which the top-dimensional simplices form a fundamental cycle under mod $2$ coefficients, and each simplex is isometric to a geodesic hyperbolic simplex. A closed hyperbolic orbifold $X$ endowed with a good triangulation is a piecewise hyperbolic pseudomanifold, where by a \textit{\textbf{good triangulation}} of $X$ we mean a triangulation of its underlying space such that all simplices are geodesic and for every dimension $\ell$, the $\ell$--stratum of the singular set of the orbifold is contained in the $\ell$--skeleton of the triangulation. 

Let $X^1$ denote the $1$--skeleton of a pseudomanifold $X$. Following Gromov and Guth in~\cite{Gromov12}, we say that an embedding of a graph $G$ into $\mathbb{R}^N$ has \textit{\textbf{combinatorial thickness}} at least $1$ if disjoint vertices and edges have disjoint $1$--neighborhoods; that is, every two distinct vertices have distance at least $2$, as do every two edges without a vertex in common and every edge and a vertex other than the edge's two endpoints.  Let $V_{1, N}(X^1)$ denote the infimum, over all embeddings of $X^1$ into $\mathbb{R}^N$ with combinatorial thickness at least $1$, of the $N$--dimensional volume of the $1$--neighborhood of the image of the embedding.  

We define the \textbf{\textit{Cheeger constant }}$h(X)$ of a closed $n$--orbifold $X$ as the greatest number such that for all open subsets $A \subseteq X$ with Hausdorff measurable boundary $\del A$, we have
\[h(X) \leq \frac{\Area \del A}{\min\{ \Vol A, \Vol X \setminus A\}}.\]

We can now state our first result. 

\begin{theorem}\label{thm-gg}
Let $X$ be a closed piecewise hyperbolic pseudomanifold of dimension $n \ge 3$, triangulated with vertex degree at most $D$.  Then for all $N \ge 3$, we have
\[V_{1, N}(X^1) \geq \const(n, N, D) \cdot \left(\frac{h(X)}{h(X) + 1} \cdot \Vol X\right)^{\frac{N}{N-1}},\]
where the constant is positive and $X^1$ denotes the $1$--skeleton of $X$.
\end{theorem}

The inequality in Theorem~\ref{thm-gg} is sharp. To show this consider a sequence $X_k$ of congruence coverings of a closed arithmetic hyperbolic orbifold $Y$ endowed with the natural triangulations obtained by lifting a fixed triangulation of $Y$. The Cheeger constants $h(X_k)$ are bounded uniformly from below by $\const(n) > 0$ (see e.g. \cite[Appendix]{Gromov12}). Therefore, we have 
\[V_{1,N}(X_k^1) \geq \const(n, N, D, Y) \cdot (\Vol X_k)^{\frac{N}{N-1}}.\]  
On the other hand, the skeletons $X_k^1$ have bounded degree of the vertices and the number of vertices proportional to $\Vol X_k$. In \cite{Kolmogorov93}, Kolmogorov and Barzdin showed that every graph $G$ with vertex degree at most $D$ admits an embedding into $\R^N$ with
\[V_{1, N}(G) \leq \const(N, D) \cdot \abs{V(G)}^{\frac{N}{N-1}},\]
where $\abs{V(G)}$ denotes the number of vertices, and so for our congruence coverings this implies that the skeletons $X_k^1$ admit embedding into $\R^N$ with 
\[V_{1,N}(X_k^1) \leq \const(N, D, Y) \cdot (\Vol X_k)^{\frac{N}{N-1}}.\] 
(Strictly speaking, Kolmogorov and Barzdin considered only the case $N = 3$ but their proof generalizes immediately to embeddings of graphs in higher dimensional spaces.)  Let us note that sharpness of the Gromov--Guth inequality for retraction thickness is not known (see a related discussion after the statement of Theorem~3.2 in \cite{Gromov12}).

Our second main result implies that the inequality in Theorem~\ref{thm-gg} is nearly sharp for a much bigger class of spaces which include any sequence of congruence arithmetic $3$--orbifolds, not necessarily covering the same space or commensurable to each other. To this end we show that arithmetic orbifolds admit good triangulations with close to optimal number of simplices.  (If the number of simplices were optimal, meaning actually proportional to the hyperbolic volume, then instead of ``nearly sharp'', we would say that Theorem~\ref{thm-gg} is sharp for these sequences of arithmetic hyperbolic orbifolds, using Kolmogorov and Barzdin's theorem.)

\begin{theorem}\label{thm-triangulate}
For any $\delta > 0$ and dimension $n = 3$, there is a constant $V_0 = V_0(\delta, n)$ such that any closed arithmetic hyperbolic $n$--orbifold of volume $\Vol(\Orb) \geq V_0$ has a good triangulation with at most $\Vol(\Orb)^{1+\delta}$ simplices and vertex degree bounded above by a constant $D = D(n)$.
\end{theorem}

The proof of Theorem~\ref{thm-triangulate} uses some deep results about volumes of arithmetic orbifolds and their relation to Lehmer's problem in number theory, to bound the injectivity radius. This approach to triangulations of hyperbolic $3$--orbifolds was first suggested in~\cite{Bel17}. A good triangulation is then obtained as a barycentric subdivision of a certain equivariant Voronoi complex in the hyperbolic $3$--space. These triangulations may have independent interest.  The proof relies on Lemma~\ref{sublem-platonic} which uses the classification of finite subgroups of $SO(3)$, but it seems plausible that both Lemma~\ref{sublem-platonic} and Theorem~\ref{thm-triangulate} are true for all $n \geq 3$.

\medskip

Combining together Theorems \ref{thm-gg} and \ref{thm-triangulate} we obtain a corollary stated as follows. 

\begin{cor}\label{thm-main}
Let $\{X_k\}_{k = 1}^\infty$ be a sequence of closed arithmetic pairwise non-isometric hyperbolic orbifolds of dimension $n = 3$ such that the Cheeger constants $h(X_k)$ are uniformly bounded below by $C > 0$.

Then for any fixed dimension $N \ge 3$ and any $\delta >0$, there exist triangulations of the underlying spaces of the orbifolds $X_k$ such that 
\begin{itemize}
 \item[(a)] the $1$--skeletons $X_k^1$ have at most $\const(n, \delta) \cdot (\Vol X_k)^{1+\delta}$ vertices; 
 \item[(b)] $X_k^1$ have a uniform bound on the number of edges at each vertex; and 
 \item[(c)] we have 
 $$\frac{V_{1, N}(X_k^1)}{\#\mathrm{vertices}(X_k^1)} \rightarrow \infty.$$
\end{itemize}
\end{cor}

\begin{proof}[Proof of Corollary~\ref{thm-main}]
First note that the assumption that the orbifolds $X_k$ are pairwise non-isometric implies that $\Vol X_k \rightarrow \infty$. This follows from the Borel--Prasad finiteness theorem applied to the arithmetic groups of isometries of the hyperbolic spaces \cite{Borel-Prasad}.

Given $N$, we choose $\delta < \frac{1}{N-1}$ and apply Theorem~\ref{thm-triangulate} to triangulate each sufficiently large $X_k$ with vertex degree at most $D$ and with 
\[\#\mathrm{simplices}(X_k) \leq (\Vol X_k)^{1+\delta}.\]
We have
\[\#\mathrm{vertices}(X_k) \leq (n+1) \cdot \#\mathrm{simplices}(X_k),\]
and thus
\[1 \geq \frac{1}{n+1} \cdot \#\mathrm{vertices}(X_k) \cdot (\Vol X_k)^{-(1+\delta)}.\]
The result is a closed piecewise hyperbolic pseudomanifold, so we may apply Theorem~\ref{thm-gg} to get
\begin{align*}
&V_{1, N}(X_k^1) \geq \const(n, N, D) \cdot \left(\frac{h(X_k)}{h(X_k) + 1} \cdot \Vol X_k\right)^{\frac{N}{N-1}} \geq \\
&\phantom{V_{1, N}(X_k^1)} \begin{aligned}\geq \const(n, N, D) & \cdot  \left(\frac{C}{C+1}\right)^{\frac{N}{N-1}} \cdot \\
& \cdot \frac{1}{n+1} \cdot \#\mathrm{vertices}(X_k) \cdot (\Vol X_k)^{\frac{N}{N-1} - (1+\delta)},
\end{aligned}
\end{align*}
and thus
\[\frac{V_{1, N}(X_k^1)}{\#\mathrm{vertices}(X_k)} \rightarrow \infty.\]
\end{proof}

This result is most relevant when the orbifolds $X_k$ all have the same underlying space, such as, for example, the sphere $S^3$.  In that case it does not appear to be known whether there exist piecewise hyperbolic pseudomanifolds whose $1$--skeletons satisfy the conclusion of the corollary.  If no such family of triangulations exists, it would imply that certain families of orbifolds, known to have Cheeger constant bounded below and to have the underlying space with bounded topology, must have only finitely many elements.  

We can draw a connection between this problem and the expander graphs. It was shown by Kolmogorov and Barzdin in the 1960's that expander graphs are hard to embed in $\R^N$ \cite{Kolmogorov93}. Hence properties (b) and (c) of Corollary~\ref{thm-main} would be satisfied if we have a sequence of triangulations of, say, $S^3$ whose $1$--skeletons form a family of expanders. Existence of such triangulations is a well known problem which has attracted considerable attention throughout the years. We can refer to Kalai's chapter~19 of the Handbook of Discrete and Computational Geometry for a related discussion \cite{Kalai}. More recently, Lackenby and Souto came up with a nice construction of such triangulations \cite{Lackenby}. Their simplicial complexes can be turned into hyperbolic orbifolds; however, these orbifolds will be non-arithmetic and, what is more essential, we expect that the number of simplices in the Lackenby--Souto triangulations would grow much faster than the volumes of the associated orbifolds. So these expander skeletons are far from optimal from our viewpoint: for them the inequality in Theorem~\ref{thm-gg} is far from sharp. The existence of triangulations of a sphere or other topological manifold whose skeletons form a sequence of \emph{geometric expanders} by satisfying the properties (a)--(c) of Corollary~\ref{thm-main} remains unknown. 

In Section~\ref{section-gg} we prove Theorem~\ref{thm-gg}, and in Section~\ref{section-triangulate} we prove Theorem~\ref{thm-triangulate}.

\medskip

\emph{Acknowledgments.}  We thank Larry Guth for bringing us together and for pointing out a mistake in an earlier version of Theorem~\ref{thm-gg}. 
We thank Marc Lackenby for explaining to us his work with Juan Souto.
We thank the referee for carefully reading the manuscript and helpful comments. 
H.~Alpert is supported by the National Science Foundation under Award No.~DMS 1802914, and M.~Belolipetsky is partially supported by CNPq, FAPERJ and Math-AmSud grants and by the MPIM in Bonn.

\section{Slicing piecewise hyperbolic manifolds}\label{section-gg}

In this section we prove Theorem~\ref{thm-gg}, based on the proof of Theorem~3.2 of~\cite{Gromov12}.  For convenience we restate it below.

\begin{reptheorem}{thm-gg}
Let $X$ be a closed piecewise hyperbolic pseudomanifold of dimension $n \ge 3$, triangulated with vertex degree at most $D$.  Then for all $N \ge 3$, we have
\[V_{1, N}(X^1) \geq \const(n, N, D) \cdot \left(\frac{h(X)}{h(X) + 1} \cdot \Vol X\right)^{\frac{N}{N-1}},\]
where the constant is positive and $X^1$ denotes the $1$--skeleton of $X$.
\end{reptheorem}

\begin{proof}
Let $i \co X^1 \hookrightarrow \mathbb{R}^N$ be an embedding of the $1$--skeleton into $\mathbb{R}^N$.  Let $N_1(X^1)$ denote the $1$--neighborhood of the image, and let $V_1(X^1)$ denote its volume.  The Falconer slicing inequality (from~\cite{Falconer80}, recalled in~\cite{Gromov12}) guarantees that we can rotate the coordinates of $\mathbb{R}^N$ to get the $x_N$ coordinate pointing in a good direction so that for every $t \in \mathbb{R}$, the $(n-1)$--dimensional volume of the slice $N_1(X^1) \cap \{x_N = t\}$ is at most $\const(N) \cdot V_1(X^1)^{\frac{N-1}{N}}$.

We view $\mathbb{R}^N$ as broken into slabs,
\[\Slab(j) = \{j \leq x_N \leq j+1\}.\]
For each $\Slab(j)$, we let $S_j$ be the subcomplex of $X$ consisting of all simplices that have a $1$--dimensional edge that intersects $\Slab(j)$.  We claim that the number of top-dimensional simplices in $S_j$ is at most $\const(n, N, D) \cdot V_1(X^1)^{\frac{N-1}{N}}$.  To show this, suppose that there are $M$ top-dimensional simplices in $S_j$.  Select one edge of each of these simplices that intersects $\Slab(j)$.  Because each edge is in at most $\binom{D-1}{n-1} = \const(n, D)$ top-dimensional simplices, after removing duplicates we have at least $\const(n, D)^{-1} \cdot M$ edges through $\Slab(j)$.  Because each edge is incident to at most $2D - 2$ other edges, we may greedily choose a subset of disjoint edges containing at least $(2D - 1)^{-1}\cdot \const(n, D)^{-1} \cdot M = \const(n, D) \cdot M$ of the original edges.  On each edge in this matching, we select a point in $\Slab(j)$; the $1$--balls around these points are disjoint and are contained in the union of $\Slab(j-1)$, $\Slab(j)$, and $\Slab(j+1)$.  From the Falconer slicing inequality, we may assume that each slab has volume at most $\const(N) \cdot V_1(X^1)^{\frac{N-1}{N}}$, so because the balls are contained in three slabs, we have
\[\Vol(\mathrm{balls}) \leq \const(N) \cdot V_1(X^1)^{\frac{N-1}{N}} \cdot 3,\]
and thus
\[M \leq \const(n, N, D) \cdot V_1(X^1)^{\frac{N-1}{N}}.\]

Next, we extend the embedding $i$ to a map
\[i \co X \rightarrow \mathbb{R}^N\]
that is smooth on each simplex---it doesn't matter whether it is an embedding---such that the image of a given simplex intersects $\Slab(j)$ only if it is in $S_j$.  We also assume that every integer $j$ is a regular value of the restriction of $x_N \circ i$ to every open simplex; by Sard's theorem this can be achieved by slightly perturbing the slab boundaries for every $j$.

The remainder of the proof is very much like the proof of Theorem~3.2 in~\cite{Gromov12}.  We let $X_j$ be the preimage $i^{-1}\Slab(j)$ in $X$, and view it as a chain in homology with coefficients in $\mathbb{Z}_2$, so that
\[[X] = \left [ \sum_{j} X_j \right].\]
We let $Z_j$ be the preimage $i^{-1}\{x_N = j\}$ in $X$, so that
\[\del X_j = Z_j + Z_{j+1}.\]
We homotope the identity map on $X$ to a map that sends each $Z_j$ to the $(n-1)$--skeleton of $X$, with the property that the image of each $X_j$ remains in $S_j$.  We can find this homotopy by choosing in each top-dimensional simplex a small ball not in any $Z_j$, and stretching that ball to cover the simplex so that the rest of the simplex maps to the boundary of the simplex.

Let $X'_j$ be the image of each $X_j$ under this homotopy.  Taking the degree mod $2$ of $X'_j$ with respect to each top-dimensional simplex, we can replace $X'_j$ by a simplicial chain $\overline{X}_j$, so that the fundamental class $[X]$ is the sum
\[[X] = \left[ \sum_j \overline{X}_j \right].\]
Similarly, we define $Z'_j$ and $\overline{Z}_j$ so that
\[\del \overline{X}_j = \overline{Z}_j + \overline{Z}_{j+1}.\]
Because each $S_j$ has at most $\const(n, N, D) \cdot V_1(X^1)^{\frac{N-1}{N}}$ top-dimensional simplices, each $\overline{X}_j$ and each $\overline{Z}_j$ has at most $\const(n, N, D) \cdot V_1(X^1)^{\frac{N-1}{N}}$ simplices also.

Notice that each $\overline{Z}_j$ is null-homologous because it is the boundary of $\sum_{i < j} \overline{X}_j$.  Thus, the definition of the Cheeger constant $h(X)$ implies that for every $\overline{Z}_j$ we can find a chain $\overline{Y}_j$ with $\del \overline{Y}_j = \overline{Z}_j$ that satisfies 
\[\Vol \overline{Y}_j \leq h(X)^{-1} \cdot \Area \overline{Z}_j.\]

In the sum
\[\sum_j (\overline{X}_j + \overline{Y}_j + \overline{Y}_{j+1}),\]
each $\overline{Y}_j$ is counted twice and cancels, so we can write the fundamental class $[X]$ as the sum of cycles
\[[X] = \sum_j [\overline{X}_j + \overline{Y}_j + \overline{Y}_{j+1}].\]
Thus, not every $\overline{X}_j + \overline{Y}_j + \overline{Y}_{j+1}$ can be null-homologous, and so at least one of them must be homologous to $[X]$ and must have total volume at least $\Vol X$.  Thus, using the fact that geodesic hyperbolic simplices have volume bounded above, for this $j$ we have
\begin{equation*}
\begin{split}
\Vol X & \leq \Vol \overline{X}_j + \Vol \overline{Y}_j + \Vol \overline{Y}_{j+1} \leq\\
& \leq \const(n, N, D) \cdot V_1(X^1)^{\frac{N-1}{N}} + 2\cdot h(X)^{-1} \cdot \const(n, N, D) \cdot V_1(X^1)^{\frac{N-1}{N}} \leq\\
& \leq \const(n, N, D) \cdot (1 + h(X)^{-1}) \cdot V_1(X^1)^{\frac{N-1}{N}}.
\end{split}
\end{equation*}
\end{proof}

When a pseudomanifold $X$ is a closed hyperbolic $n$--manifold both our Theorem~\ref{thm-gg} and the Gromov--Guth Theorem~3.2 can be applied to it, and it would be interesting to compare the results. This leads to a question about the relation between combinatorial thickness and the retraction thickness from \cite{Gromov12}. We recall the definitions:

\begin{definition}
A manifold $X$ embedded in $\R^N$ is said to have \textit{\textbf{retraction thickness}} at least $T$ if the $T$--neighborhood of $X$ retracts to $X$. 
\end{definition}

\begin{definition}
A pseudomanifold $X$ whose $1$--skeleton is embedded in $\R^N$ with combinatorial thickness $T$ is said to have \textit{\textbf{thickness}} at least $T$.
\end{definition}

Given a subset $Y \subset \R^N$, denote by $V_{T}(Y)$ the $N$--dimensional volume of its $T$--neighborhood. Now assume that a closed hyperbolic manifold of dimension $n \ge 3$ has an embedding $i \co X \hookrightarrow \R^N$ with retraction thickness $T$. One can then try to construct a triangulation of the image $i(X)$ whose $1$--skeleton has a combinatorial thickness $T$ (or at least $T-\varepsilon$ for an arbitrary small $\varepsilon > 0$). If there is such a triangulation, then we can apply the simplex straightening to its simplices and obtain a piecewise hyperbolic pseudomanifold isometric to $X$ such that $V_{T-\varepsilon}(X^1) \le V_{T}(i(X))$. It would then allow us to deduce Theorem~3.2 from \cite{Gromov12} from our Theorem~\ref{thm-gg}.  

Reciprocally, suppose that we have an embedding $\iota \co X^1 \hookrightarrow \R^N$ with combinatorial thickness $T$. Assuming that the codimension is large compared to $n$, we can extend it to an embedding $\tilde{\iota} \co X \hookrightarrow \R^N$. Can this embedding have retraction thickness $T'$ close to $T$ and the volume $V_{T'}(\tilde{\iota}(X))$ bounded in terms of $V_{T}(\iota(X^1))$? If yes, this would have implications for sharpness of the inequality from Theorem~3.2 in \cite{Gromov12}. 

We leave these questions for future research.

\section{Triangulating arithmetic hyperbolic orbifolds}\label{section-triangulate}

In this section we prove Theorem~\ref{thm-triangulate} restated below. 

\begin{reptheorem}{thm-triangulate}
For any $\delta > 0$ and dimension $n = 3$, there is a constant $V_0 = V_0(\delta, n)$ such that any closed arithmetic hyperbolic $n$--orbifold of volume $\Vol(\Orb) \geq V_0$ has a good triangulation with at most $\Vol(\Orb)^{1+\delta}$ simplices and vertex degree bounded above by a constant $D = D(n)$.
\end{reptheorem}



Arithmeticity of the orbifolds is essential for Theorem~\ref{thm-triangulate}. We begin with recalling the definition of arithmetic subgroups.  Let $H$ be a linear semisimple Lie group with trivial center and let $\mathrm{G}$ be an algebraic group defined over a number field $k$ such that $\mathrm{G}(k \otimes_\Q \R)$ is isogenous to $H\times K$, where $K$ is a compact Lie group. Consider a natural projection $\phi \co \mathrm{G}(k \otimes_\Q \R) \to H$. The image of the group of $k$--integral points $\phi(\mathrm{G}(\mathcal{O}_k))$ and all subgroups of $H$ which are commensurable with it are called \textit{\textbf{arithmetic subgroups of $H$ defined over $k$}}. Arithmetic subgroups are lattices, i.e., they are discrete and have finite covolume in $H$. Their quotient spaces are called \textit{\textbf{arithmetic orbifolds}}. In our case, $H = \mathrm{PO}(n,1)$ is the group of isometries of the hyperbolic space $\Hy^n$ and the quotient orbifolds are hyperbolic $n$--orbifolds. We refer to \cite{WitteMorris} for a comprehensive introduction to the theory of arithmetic subgroups.  

Let $\Orb =  \Hy^n/\Gamma$ be a closed hyperbolic orbifold with singular set $\Sigma$, and let $\pi \co \Hy^n \to \Orb$ be the covering map. The elements of the group $\Gamma$ fall into two types: \textit{\textbf{elliptic}} are those which have fixed points in $\Hy^n$ and \textit{\textbf{hyperbolic}} are those which act freely. For a hyperbolic isometry $\gamma \in \Gamma$ its \textit{\textbf{displacement at $x \in \Hy^n$}} is defined by $\ell(\gamma, x) = \dist(x, \gamma x)$ and the \textit{\textbf{displacement}} of $\gamma$ (also called its \textit{translation length}) is 
\[\ell(\gamma) = \inf_{x \in \Hy^n} \ell(\gamma, x).\]
It is equal to the displacement of $\gamma$ at the points of its axis. We will define the orbifold \textit{\textbf{injectivity radius}} by $r_{inj}(\Orb) = \inf \{\frac12 \ell(\gamma)\}$, where the infimum is taken over all hyperbolic elements $\gamma \in \Gamma$. It is equal to half of the smallest length of a closed geodesic in $\Orb$. When $\Orb$ is a manifold, this definition is equivalent to the usual definition of the injectivity radius as the supremum of $r$ such that any point $p\in\Orb$ admits an embedded ball $B(p, r) \subset \Orb$. This is not the case in general; the points in the singular set only admit embedded folded balls (see \cite{Samet13} for the definition of folded balls).

We will first assume that $r_{inj}(\Orb) \ge r > 0$ and that any finite subgroup $F < \Gamma$ has order $|F| \le q$. A similar problem was considered before by Gelander and Samet (see \cite[Section~2]{BGLS10} and \cite{Samet13}). The difference in our case is that we require an explicit control over the constants and that we want to construct a good triangulation of $\Orb$, not just a simplicial complex homotopy equivalent to it. 

By the Margulis lemma there exist constants $\mu_n > 0$ and $m_n \in \N$ depending only on the dimension $n$, such that any subgroup of $\Gamma$ generated by the elements whose displacements at some point $x$ are bounded above by $\mu_n$ contains a normal nilpotent subgroup of index at most $m_n$. We refer to \cite[Theorem~8.3]{Gromov-lect} for the general statement and the proof of the lemma. We will use this result to obtain certain constraints on the position of the singular set in $\Orb$. 

\begin{lemma}\label{lemma-Voronoi}
For $n = 3$ let $\Orb = \mathbb{H}^n/\Gamma$ be a closed hyperbolic orbifold, and let $\varepsilon = \min \{\frac{\mu_n}{8}, \frac{r}{16m_n}\}$, where $\mu_n$ and $m_n$ are dimensional constants arising from the Margulis lemma, and $r \leq r_{inj}(\Orb)$.  Then there is a good triangulation $T$ of $\Orb$ such that the vertex degree is bounded by $D(n)$ and the number of simplices is bounded by $C(n)\frac{q\Vol(\Orb)}{v_\varepsilon}$, where $C(n)$ and $D(n)$ are dimensional constants, $q$ is the maximum size of a finite subgroup $F < \Gamma$, and $v_{\varepsilon}$ denotes the volume of a ball of radius $\varepsilon$ in $\mathbb{H}^n$.
\end{lemma}

\begin{proof}
Let $S$ be any maximal $2\varepsilon$--separated set of points in $\Orb$ that are not in the singular set, let $\overline{S}$ be the set of lifts of those points in $\mathbb{H}^n$, and let $P$ be the Voronoi decomposition of $\mathbb{H}^n$ corresponding to $\overline{S}$.  It is a cell decomposition with one top-dimensional cell for each point of $\overline{S}$, and this top-dimensional cell is equal to the convex hyperbolic polytope consisting of all points of $\mathbb{H}^n$ that are closer to our selected point than to any other point of $\overline{S}$.

We define a barycentric subdivision of $P$ as follows.  For any convex hyperbolic polytope, there is a unique point that minimizes the sum of squared distances to the vertices of the polytope; this is because squared distance to a point is a strictly convex function on $\mathbb{H}^n$ \cite[Theorem 4.1(2)]{Bishop69}.  We refer to this point as the \textbf{\textit{barycenter}} of the polytope.  The barycenter is in the relative interior of the polytope, because for every point, the negative gradient of the sum of squared distances to the vertices is a sum of vectors pointing toward the vertices.  Thus, we can form a triangulation $\overline{T}$ of $\mathbb{H}^n$, in which the vertices are the barycenters of all the faces of $P$ of all dimensions, and the simplices (all equal to the convex hulls of their vertices) correspond to chains of faces of $P$, under the partial ordering by inclusion of closures.  Because $P$ is $\Gamma$--invariant, so is $\overline{T}$, and so we can set $T$ to be the triangulation of $\Orb$ corresponding to $\overline{T}$.

First we check that $T$ is a good triangulation, that is, that for every dimension $\ell$, the $\ell$--stratum of the singular set of $\Orb$ is contained in the $\ell$--skeleton of $T$.  Let $x \in \Orb$ be any point, and let $d$ be the least dimension of any simplex of $T$ containing $x$.  Consider the stabilizer in $\Gamma$ of any lift $\overline{x}$ of $x$.  Any $g \in \Gamma$ that fixes $\overline{x}$ must send the whole $d$--simplex containing $\overline{x}$ to itself.  But the $d+1$ vertices of this simplex all come from different-dimensional faces of $P$, so $g$ cannot permute them in any way other than by the identity.  Thus the whole $d$--simplex is in the fixed-point set of the stabilizer of $\overline{x}$, and so if $x$ is in the $\ell$--stratum, then $\ell \geq d$.

Next we check that there is a bound on the vertex degree that depends only on the dimension $n$.  For $i = 0, 1, \ldots, n$, let $P_i$ be the set of vertices of $T$ that are the images in $\Orb$ of barycenters of $i$--dimensional faces of $P$.  First, for each vertex $v \in P_n$, let us bound the number of neighbors of $v$ in $P_{n-1}$.  This is equivalent to counting top-dimensional cells in $P$ that neighbor the cell of a lift of $v$.  Let $\overline{x}$ be the point of $\overline{S}$ that corresponds to a lift $\overline{v}$ of $v$, and let $\overline{y}_1, \ldots, \overline{y}_k$ be the points of $\overline{S}$ such that the cells of $\overline{y}_1, \ldots, \overline{y}_k$ share an $(n-1)$--dimensional face with the cell of $\overline{x}$.  Each $\overline{y}_i$ is within $4\varepsilon$ of $\overline{x}$.  Suppose first that the projections of $\overline{x}, \overline{y}_1, \ldots, \overline{y}_k$ to $\Orb$ are distinct.  In this case the $\varepsilon$--balls around $\overline{x}, \overline{y}_1, \ldots, \overline{y}_k$ are all disjoint.  The number of disjoint $\varepsilon$--balls that can fit within $4\varepsilon$ of a given point in $\mathbb{H}^n$ is monotonic in $\varepsilon$, so because we have assumed $\varepsilon \leq \frac{\mu_n}{8}$, we have a dimensional upper bound on $k$ in this case of disjoint projections.

The projections of $\overline{x}, \overline{y}_1, \ldots, \overline{y}_k$ to $\Orb$ may not all be distinct; that is, the cell of $v$ may be adjacent to itself one or more times, or may be adjacent to another cell multiple times.  We need to bound these multiplicities.  First we claim that $\varepsilon$ has been chosen such that if a $4\varepsilon$--ball in $\mathbb{H}^n$ contains several points of the same $\Gamma$--orbit, then there is a finite subgroup $H$ of $\Gamma$ such that these points are in the same $H$--orbit.  The constants $\mu_n$ and $m_n$ in the Margulis lemma have the following property.  For any $p \in \mathbb{H}^n$ and any $t \in \mathbb{R}$, let $\Gamma_t(p)$ denote the subgroup of $\Gamma$ generated by the elements that move $p$ by distance less than $t$.   Then if $t \leq \mu_n$ and if $\Gamma_t(p)$ is infinite, there is an element in $\Gamma$ of infinite order that moves $p$ by distance less than $2m_nt$ \cite[Lemma~2.3]{Samet13}.  We have chosen $\varepsilon$ such that $8\varepsilon \leq \mu_n$ and $2m_n(8\varepsilon) \leq r \leq r_{inj}(\Orb)$.  By definition, every element in $\Gamma$ of displacement less than $r_{inj}(\Orb)$ has a fixed point and therefore has finite order.  Thus, $\Gamma_{8\varepsilon}(p)$ must be finite for all $p \in \mathbb{H}^n$.  Let $p_1, \ldots, p_k$ be points in some $4\varepsilon$--ball in $\mathbb{H}^n$ that all map to the same point of $\Orb$.  Then they are all in the orbit of $p_1$ under $\Gamma_{8\varepsilon}(p_1)$, which we choose to be our finite subgroup $H$.

Suppose that one or more of the neighbors $\overline{y}_1, \ldots, \overline{y}_k$ of $\overline{x}$ project to the same point $x \in \Orb$ as $\overline{x}$ does.  From the previous paragraph we know that all such points are in the orbit of a finite subgroup $H$, and all elements of $H$ have a common fixed point.  Using the assumption that we are in dimension $n = 3$, we may apply Lemma~\ref{sublem-platonic} below to get a uniform bound on the number of neighbors $\overline{y}_1, \ldots, \overline{y}_k$ in the orbit of $\overline{x}$.  Similarly, suppose that one or more of $\overline{y}_1, \ldots, \overline{y}_k$ project to the same point $y_1 \in \Orb$ as $\overline{y}_1$ does, distinct from $x$.  These points are in the orbit of a finite subgroup $H$, and to bound how many of them may be neighbors of $\overline{x}$, we examine the Voronoi decomposition of $H\overline{x} \cup H\overline{y}_1$ and apply Lemma~\ref{sublem-platonic} to get a uniform bound.

We have a bound on how many $\overline{y}_1, \ldots, \overline{y}_k$ with distinct projections to $\Orb$ can be neighbors of $\overline{x}$, and in the case of dimension $n = 3$ we have a bound on the multiplicity with which they have the same projections as either $\overline{x}$ or each other.  In total this gives a bound $d_1(n)$ on the number of $P_{n-1}$--neighbors of each $v \in P_n$ in the case $n = 3$.

Then, we can use this bound to bound the total number of neighbors of each vertex $v \in P_n$.  Consider the cell in $P$ of a lift $\overline{v}$ of $v$.  The point $s$ in $S$ corresponding to this cell is not in the singular set of $\Orb$, and we claim that this implies that the interior of this cell maps injectively to $\Orb$.  Suppose to the contrary that some nontrivial element $g$ of $\Gamma$ takes this cell to itself.  Then it fixes the barycenter $\overline{v}$ of the cell but must move the lift $\overline{s}$ of $s$ because $s$ is not in the singular set, but this means that $\overline{v}$ is equidistant between $\overline{s}$ and $g\overline{s}$, contradicting the definition of the Voronoi decomposition because we know that $\overline{v}$ is in the interior of the cell.  Thus every top-dimensional cell in $P$ maps injectively to $\Orb$.  

This implies that when the closure of the cell of $\overline{v}$ in $P$ is mapped to $\Orb$, the $(n-1)$--dimensional faces are identified in at most pairs; no three $(n-1)$--dimensional faces can be identified, because the nearby parts of the interior of the cell do not get identified.  Thus, the total number of $(n-1)$--dimensional faces of the cell of $\overline{v}$ is at most $2d_1(n)$.  Every subset of $(n-1)$--dimensional faces intersects in at most one arbitrary-dimensional face of the cell of $\overline{v}$, so the total number of faces of the cell of $\overline{v}$ is at most $2^{2d_1(n)}$, and thus the total degree of $v$ is at most $2^{2d_1(n)}$.

Similarly, if instead we let $v$ be a vertex in any $P_i$, we can bound the number of adjacencies to vertices in $P_n$.  Counting with multiplicity is a little tricky here.  If $\overline{v}$ is a lift of $v$, and $\overline{y}_1$ and $\overline{y}_2$ are the points of $\overline{S}$ corresponding to cells that have $\overline{v}$ as a boundary point, then the segments from $\overline{v}$ to $\overline{y}_1$ and $\overline{y}_2$ give the same edge in $T$ if some element of $\Gamma$ takes $\overline{y}_1$ to $\overline{y}_2$ while fixing $\overline{v}$; otherwise, the two segments give two different edges in $T$.  Let $\overline{y}_1, \ldots, \overline{y}_k$ be the points of $\overline{S}$ corresponding to all of the cells that have $\overline{v}$ as a boundary point.  They are the closest points in $\overline{S}$ to the point $\overline{v}$, so they are within $2\varepsilon$ of $\overline{v}$.

We can bound the number of distinct projections of $\overline{y}_1, \ldots, \overline{y}_k$ to $\Orb$ because their $\varepsilon$--balls in $\mathbb{H}^n$ are disjoint and so we can take the minimum number of balls that fit when $\varepsilon = \frac{\mu_n}{8}$.  Next we need to bound the multiplicity with which $v$ may have different adjacencies to the same projection to $\Orb$; to do this, we use $n = 3$ and apply Lemma~\ref{sublem-platonic} to $H\overline{v} \cup H\overline{y}_1$, where $H$ is the finite subgroup of $\Gamma$ taking $\overline{y}_1$ to all other $\overline{y}_i$ that are in its $\Gamma$--orbit.  The case where $\overline{v}$ is fixed by $H$ does not give rise to different adjacencies, so Lemma~\ref{sublem-platonic} gives a bound on the number of different adjacencies from $v$ to any vertex in $P_n$.  Putting the bounds together, for $n =3$ we get a dimensional upper bound $d_2(n)$ on the number of $P_{n}$--neighbors of each $v \in P_i$.

We can bound the vertex degree of $T$ using the bound on the number of neighbors in $P_i$ of each element of $P_n$ and the bound on the number of neighbors in $P_n$ of each element of $P_i$.  Given any vertex $v \in P_i$, if $u$ is any neighbor of $v$, then $u$ and $v$ have a common neighbor $w \in P_n$.  Thus, the total number of neighbors of $v$ is at most $d_2(n) \cdot 2^{2d_1(n)}$, and so we set $D(n) = d_2(n) \cdot 2^{2d_1(n)}$.

Finally, we prove the bound on the number of top-dimensional simplices in $T$.  Each simplex has one vertex in $P_n$, and each vertex in $P_n$ is in at most $\binom{D(n)}{n}$ simplices, so the total number of simplices is at most $\binom{D(n)}{n}$ times the number of points in our original $2\varepsilon$--separated set $S$, and therefore it suffices to show that
\[\abs{S} \leq \frac{q \Vol(\Orb)}{v_\varepsilon}.\]
To show this, we claim that every $\varepsilon$--ball in $\mathbb{H}^n$ maps to $\Orb$ with multiplicity at most $q$ at each point.  This is because we have shown above that if $p_1, \ldots, p_k \in \mathbb{H}^n$ are in the same $4\varepsilon$--ball and also in the same $\Gamma$--orbit, they are also in the same $H$--orbit for some finite subgroup $H$ of $\Gamma$.  Because we have assumed that every finite subgroup of $\Gamma$ has at most $q$ elements, we must have $k \leq q$.

Thus, the $\varepsilon$--balls around the points of $S$ are disjoint in $\Orb$ and each has volume at least $\frac{v_\varepsilon}{q}$, where $v_{\varepsilon}$ denotes the volume of a ball of radius $\varepsilon$ in $\mathbb{H}^n$.  In total, the volume is at most $\Vol(\Orb)$, so we have $\abs{S} \leq \frac{q\Vol(\Orb)}{v_\varepsilon}$, and thus 
\[\#\mathrm{simplices}(T) \leq \binom{D(n)}{n} \cdot \frac{q\Vol(\Orb)}{v_\varepsilon} = C(n) \cdot \frac{q\Vol(\Orb)}{v_{\varepsilon}}.\]
\end{proof}

The proof above relies on the following additional lemma to bound the degree of the triangulation that arises from the Voronoi decomposition.  Although this lemma seems like it may be true more generally, we only know how to prove it in $3$ dimensions.  Proving this lemma is the only part of this paper where the assumption $n = 3$ is needed.

\begin{lemma}\label{sublem-platonic}
There is a constant $M$ such that the following is true.  Let $H$ be a finite group of rotations of $\mathbb{H}^3$ with a common fixed point, $*$.  Let $p$ and $q$ be points not fixed by any nontrivial elements of $H$.  Consider the Voronoi decomposition corresponding to the set $Hp \cup Hq$.  Then each $3$--dimensional cell has at most $M$ $2$--dimensional facets.
\end{lemma}

\begin{proof}
The finite groups of rotations of $S^2$ are classified: $H$ must be either a cyclic group, a dihedral group, or a group of rotations of a Platonic solid.  If $H$ is a group of rotations of a Platonic solid, we have a uniform bound on $\abs{H}$ and thus on the number of $3$--cells.  Two $3$--cells share at most one $2$--dimensional facet, so there are at most $2\abs{H} - 1$ facets per $3$--cell in this case.

Suppose that $H$ is cyclic, and that the common axis of the rotations is vertical.  We consider separately the Voronoi decomposition corresponding to $Hp$ and the Voronoi decomposition corresponding to $Hq$.  If two cells of the $(Hp \cup Hq)$--decomposition share a facet, then either they correspond to two elements of $Hp$ that have adjacent cells in the $Hp$--decomposition, or they correspond to two elements of $Hq$ that have adjacent cells in the $Hq$--decomposition, or they correspond to one element of $Hp$ and one element of $Hq$.

To see how many $Hp$--neighbors an element of $Hp$ can have, we observe that the $Hp$--decomposition looks like $\abs{H}$ congruent vertical wedges, so each wedge has two neighbors.  Similarly each element of $Hq$ has at most two $Hq$--neighbors.

Suppose that two cells are neighbors, one from an element of $Hp$ and the other from an element of $Hq$.  Without loss of generality, suppose that these elements are $p$ and $q$.  Then on the facet between the two cells, each point of the facet is closer to $p$ than to any other point of $Hp$, and it is closer to $q$ than to any other point of $Hq$.  Thus, in the $Hp$--decomposition, this point is in the cell of $p$, and in the $Hq$--decomposition, it is in the cell of $q$.  By examining the geometry of the congruent vertical wedges, we can see that the cell of $p$ in the $Hp$--decomposition intersects two cells of the $Hq$--decomposition, unless the $Hp$-- and $Hq$--decompositions are identical, in which case it intersects only one cell of the $Hq$--decomposition.  Thus, in the $(Hp \cup Hq)$--decomposition, the cell of $p$ can neighbor at most two cells of points in $Hq$.

In total, each cell of the $(Hp \cup Hq)$--decomposition can neighbor at most two cells of its own type and at most two cells of the other type, for a total of at most four cells in the case where $H$ is cyclic.

The argument when $H$ is dihedral is very similar.  Suppose that $H = D_{2k}$, and that in its cyclic subgroup $C_k$ the common axis of the rotations is vertical.  The unit sphere around the fixed point $*$ has a north pole and a south pole on this vertical axis.  Looking at the Voronoi decompositions with respect to $Hp$ and $Hq$ separately, we see that each has $k$ wedges touching the north pole, and $k$ wedges touching the south pole, rotated from each other by some offsets depending on $p$ and $q$.  The cell in the $Hp$--decomposition containing $p$ has at most four neighbors from $Hp$: two touching the same pole, and at most two touching the other pole.  It also intersects at most four cells in the $Hq$--decomposition: at most two touching the same pole, and at most two touching the other pole.  Thus, in the $(Hp \cup Hq)$--decomposition, each cell can neighbor at most eight cells in the case where $H$ is dihedral.
\end{proof}

This completes the proof of Lemma~\ref{lemma-Voronoi}.  We now bring in the arithmetic information for estimating the number of simplices in terms of volume. 

Given an integral monic polynomial $P(x)$ of degree $d$, its \textit{\textbf{Mahler measure}} is defined by 
\[ M(P) = \prod_{i=1}^d \mathrm{max}(1, |\theta_i|),\]
where $\theta_1$,\ldots, $\theta_d$ are the roots of $P(x)$.

Let $\gamma \in \Gamma$ be a hyperbolic transformation. By \cite[Proposition~1(1,4)]{Greenberg62}, the eigenvalues of $\gamma$ considered as an element of $\mathrm{O}(n,1)$ are $e^{\pm \ell(\gamma)}$ and $n-1$ eigenvalues whose absolute value is $1$. We would like to relate $e^{\ell(\gamma)}$ to the Mahler measure of a certain polynomial naturally associated to $\gamma$. To this end we can adapt the argument of \cite[Section~10]{Gelander04}. Let $H^\circ$ be the identity component of the group $H = \isom(\Hy^n)$. It is center-free and connected so we can identify it with its adjoint group $\mathrm{Ad}(H^\circ) \le \mathrm{GL}(\mathfrak{g})$, where $\mathfrak{g}$ denotes the Lie algebra of $H$. We have $\Gamma' = \Gamma\cap H^\circ$, a cocompact arithmetic lattice, and $\gamma^2 \in \Gamma'$. Since $\Gamma'$ is arithmetic, there is a compact extension $H^\circ\times K$ of $H^\circ$ and a $\Q$--rational structure on the Lie algebra $\mathfrak{g}\times\mathfrak{k}$ of $H^\circ\times K$, such that $\Gamma$ is the projection to $H^\circ$ of a lattice $\tilde{\Gamma}$, which is contained in $(H^\circ\times K)_\Q$ and commensurable to the group of integral points $(H^\circ \times K)_\Z$ with respect to some $\Q$--base of $(\mathfrak{g}\times\mathfrak{k})_\Q$. By changing this $\Q$--base we can assume that $\tilde{\Gamma}$ is contained in $(H^\circ \times K)_\Z$. This means that the characteristic polynomial $P_{\tilde{\gamma}}$ of any $\tilde{\gamma} \in \tilde{\Gamma}$ is a monic integral polynomial of degree $(n+1)\deg(k)$, where $k$ is the field of definition of the arithmetic group. Since $K$ is compact, any eigenvalue of $\tilde{\gamma}$ with absolute value different from $1$ is also an eigenvalue of its projection in $H^\circ$. Therefore, 
\begin{equation*}
e^{\ell(\gamma^2)} = M(P_{\tilde{\gamma}^2}); 
\end{equation*}

\begin{equation}\label{sec2:eq1}
\ell(\gamma) \ge \frac12 \log M(P_{\tilde{\gamma}^2}).
\end{equation}
This implies that $r_{inj}(\Orb) \ge \min\{\frac14 \log M(P_{\tilde{\gamma}}) \}$, where the minimum is taken over all $\tilde{\gamma} \in \tilde{\Gamma}$ which project to hyperbolic elements in $\Gamma'$. Moreover, our argument shows that the degrees of the irreducible integral monic polynomials whose Mahler measures appear in this bound satisfy
\begin{equation}\label{sec2:eq2}
d \le (n+1)\deg(k).
\end{equation}

Let us mention in passing that more precise versions of inequalities \eqref{sec2:eq1} and \eqref{sec2:eq2} for arithmetic subgroups of the simplest type were obtained in \cite{ERT16}. 

Now recall that the celebrated Lehmer's problem says that the Mahler measures of non-cyclotomic polynomials are expected to be uniformly bounded away from~$1$. A special case of this conjecture also known as the Margulis conjecture implies a uniform lower bound for the lengths of closed geodesics of arithmetic locally symmetric $n$--dimensional manifolds (see \cite[Section~10]{Gelander04}). These conjectures have attracted a lot of interest but still remain wide open. Nevertheless, there are some quantitative number-theoretic results towards Lehmer's problem which we can use for our estimates.

In \cite{Dobr79}, Dobrowolski proved the following lower bound for the Mahler measure:
\begin{equation}\label{sec2:eq3}
\log M(P) \ge c_1\left(\frac{\log\log d}{\log d}\right)^3,
\end{equation}
where $d$ is the degree of the polynomial $P$ and $c_1>0$ is an explicit constant.

We can relate the degree $d$ to the volume by using an important inequality relating the volume of a closed arithmetic orbifold and the degree of its field of definition:
\begin{equation}\label{sec2:eq4}
\mathrm{deg}(k) \le c_2\log\Vol(\Orb) + c_3.
\end{equation}
For hyperbolic orbifolds of dimension $n \ge 4$ this inequality follows from \cite[Section~3.3]{Bel07} and Minkowski's bound for discriminant. In dimensions $2$ and $3$ this inequality is a result of Chinburg and Friedman \cite{CF86}, and in the form stated here it can be found in \cite[Section~3]{BGLS10}. 

For sufficiently large $x$ the function $\frac{\log x}{x}$ is monotonically decreasing, hence for sufficiently large volume we obtain
\begin{equation}\label{sec2:eq5}
r_{inj}(\Orb) \ge \frac{c_1}4\left(\frac{\log\log\log \Vol(\Orb)^c}{\log\log \Vol(\Orb)^c}\right)^3.
\end{equation}

We note that this is a very slowly decreasing function. 

Next we need to bound the order $q$ of finite subgroups $F < \Gamma$ in terms of volume. This can be done using the Margulis lemma once again, this time applied to the discrete subgroups of $\mathrm{O}(n)$. The details for arithmetic subgroups of the simplest type can be found in \cite[Lemma~4.4 and Corollary~4.5]{ABSW08}. A similar argument applies in general: Consider a $k$--embedding of $\Gamma$ into $\mathrm{GL}(m, k)$ with $m = n+1$ if $n$ is even, $m = 2(n+1)$ if $n$ is odd and $\neq 7$ (cf. \cite[Proposition~6.4.8]{WitteMorris}), and $m = 24$ if $n =7$. The last embedding comes from the fact that an adjoint simple group of type $\mathrm{D}_4$ over $k$ is a connected component of the automorphism group of a trialetarian algebra \cite[Chapter~X]{Book_of_involutions}. Starting from this place we can repeat the proof of the lemma and the corollary cited above. The resulting inequality is 
\begin{equation}
 q \le c_4\deg(k)^{c_5}, 
\end{equation}
with the constants $c_4$, $c_5 > 0$ depending only on $n$. 

Together with \eqref{sec2:eq4} it implies 
\begin{equation}\label{sec2:eq6}
 q \le c_6(\log\Vol(\Orb))^{c_5}.
\end{equation}

It remains to apply inequalities \eqref{sec2:eq5} and \eqref{sec2:eq6} for estimating the number of simplices in a good triangulation provided by Lemma~\ref{lemma-Voronoi}. For sufficiently large volume, we have
\begin{align*}
&r_{inj}(\Orb) \geq \frac{c_1}{4}\left(\frac{\log\log\log \Vol(\Orb)^c}{\log\log \Vol(\Orb)^c}\right)^3 \geq\\
&\phantom{r_{inj}(\Orb)} \geq \frac{c_1}{4} \left(\frac{1}{\log\log \Vol(\Orb)^c}\right)^3 \geq \\
&\phantom{r_{inj}(\Orb)} \geq \frac{c_1}{4}\left(\frac{1}{\log \Vol(\Orb)}\right)^3,
\end{align*}
so for $\varepsilon = \min\left\{\frac{\mu_n}{8}, \frac{r}{16m_n}\right\}$ we have
\[v_{\varepsilon} \geq C(n) \cdot \varepsilon^n \geq C(n) \cdot \left(\frac{1}{\log\Vol(\Orb)}\right)^{3n}.\]
Thus for sufficiently large volume we obtain
\begin{align*}
&\#\mathrm{\ simplices}(T) \leq C(n) \cdot q \cdot \Vol(\Orb) \cdot \frac{1}{v_{\varepsilon}} \leq \\
&\phantom{\#\mathrm{\ simplices}} \leq C(n) \cdot (\log \Vol(\Orb))^{c_5} \cdot \Vol(\Orb) \cdot (\log \Vol(\Orb))^{3n} \leq \\
&\phantom{\#\mathrm{\ simplices}} \leq \Vol(\Orb)^{1+\delta}.
\end{align*}
This finishes the proof of the theorem. \qed

\begin{rem} Let us note that even assuming the solution to Lehmer's problem or the Margulis conjecture our method would not allow us to deduce a better bound for the size of a good triangulation. This is because of the contribution of the singularities of large order to the volume estimates. A more careful analysis under this assumption may allow one to produce a linear upper bound but we will not pursue it here. 
\end{rem}

\bibliographystyle{amsalpha}
\bibliography{orb-exp-bib}{}
\end{document}